\documentclass[a4paper,12pt]{amsart}

\usepackage{amsfonts}
\usepackage{amsmath}
\usepackage{amssymb}
\usepackage{graphicx}

\usepackage[usenames]{color}
\usepackage[colorlinks]{hyperref}

\newtheorem{thm}{Theorem}[section]
\newtheorem{lem}[thm]{Lemma}
\newtheorem{prop}[thm]{Proposition}

\theoremstyle{definition}

\theoremstyle{remark}
\newtheorem{rem}{Remark}[section]
\newtheorem{defn}{Definition}

\numberwithin{equation}{section}

\def\d{\mathrm d}

\let\Re\relax
\DeclareMathOperator{\Re}{Re}

\begin{document}

\title[On a CNLSE with irregular coefficients]{On a fractional nonlinear Schr\"odinger equation with irregular coefficients. case: $d<2s$}

\author[A. Altybay]{Arshyn Altybay}
\address{
  Arshyn Altybay:
    \endgraf
  Department of Mathematics: Analysis, Logic and Discrete Mathematics
  \endgraf
  Ghent University, Krijgslaan 281, Building S8, B 9000 Ghent
  \endgraf
  Belgium, 
  \endgraf
   Institute of Mathematics and Mathematical Modeling
  \endgraf
  125 Pushkin str., 050010 Almaty, Kazakhstan
  \endgraf  
    and
  \endgraf 
    al--Farabi Kazakh National University
  \endgraf
  71 al--Farabi ave., 050040 Almaty, Kazakhstan
  \endgraf
    {\it E-mail address} {\rm arshyn.altybay@gmail.com, arshyn.altybay@ugent.be}
 }

\author[M. Ruzhansky]{Michael Ruzhansky}
\address{
  Michael Ruzhansky:
  \endgraf
  Department of Mathematics: Analysis, Logic and Discrete Mathematics
  \endgraf
  Ghent University, Krijgslaan 281, Building S8, B 9000 Ghent
  \endgraf
  Belgium
  \endgraf
  and
  \endgraf
  School of Mathematical Sciences
  \endgraf
  Queen Mary University of London
  \endgraf
  United Kingdom
  \endgraf
  {\it E-mail address} {\rm michael.ruzhansky@ugent.be}
}

\author[M. Sebih]{Mohammed Elamine Sebih}
\address{
  Mohammed Elamine Sebih:
  \endgraf
  Laboratory of Geomatics, Ecology and Environment (LGEO2E)
  \endgraf
  Mustapha Stambouli University of Mascara, 29000 Mascara
  \endgraf
  Algeria
  \endgraf
  {\it E-mail address} {\rm sebihmed@gmail.com, ma.sebih@univ-mascara.dz}
}

\author[N. Tokmagambetov ]{Niyaz Tokmagambetov}
\address{
  Niyaz Tokmagambetov:
  \endgraf 
    al--Farabi Kazakh National University
  \endgraf
  71 al--Farabi ave., 050040 Almaty, Kazakhstan
  \endgraf
  and
  \endgraf   
  Institute of Mathematics and Mathematical Modeling
  \endgraf
  125 Pushkin str., 050010 Almaty, Kazakhstan
  \endgraf  
  {\it E-mail address:} {\rm tokmagambetov@math.kz}
  }

\thanks{This research was funded by the Science Committee of the Ministry of Education and Science of the Republic of Kazakhstan (Grant No. AP27508473), by the FWO Odysseus 1 grant G.0H94.18N: Analysis and Partial Differential Equations, and by the Methusalem programme of the Ghent University Special Research Fund (BOF) (Grant number 01M01021). MR is also supported by EPSRC grant UKRI3645 and by FWO grant G083525N.}


\begin{abstract}

In the case when $d<2s$, where $d$ is the space dimension and $s$ is the fractional power of the Laplacian, we study the well-posedness for a cubic nonlinear Schrödinger equation (CNLSE) generated by the fractional Laplacian and involving distributional, or less regular, coefficients. We formulate our problem in the setting of the concept of so-called very weak solutions and prove that it has a very weak solution. Moreover, we prove the uniqueness in some adequate sense as well as the compatibility of the very weak solution with the classical one when the latter exists. Our results cover the classical case when: $d=1, s=1$. A second task in this paper is to conduct some numerical experiments where interesting behaviours of the very weak solution are observed. The obtained result is the first example of the very weak well-posedness in the setting of nonlinear partial differential equations.
\end{abstract}

\maketitle


\section{Introduction}
This article investigates the Cauchy problem for the cubic nonlinear Schr\"odinger equation (CNLSE) generated by the fractional Laplacian. More precisely, for fixed $T>0$ and $s> 0$, we consider:
\begin{equation}\label{Equation in introduction}
    \left\lbrace
    \begin{array}{l}
    iu_{t}(t,x) = (-\Delta)^{s} u(t,x) + V(x)u(t,x) + g(x)\vert u(t,x)\vert^2 u(t,x),\\
    u(0,x)=u_{0}(x),
    \end{array}
    \right.
\end{equation}
where $(t,x)\in\left[0,T\right]\times \mathbb{R}^{d}$, and $V$ and $g$ are assumed to be real-valued and non-negative. We allow the equation coefficients $V$ and $g$, and the Cauchy data $u_0$ to be distributions or more singular, and we prove that the problem is very weakly well-posed in the case when $d<2s$.

In the particular case when $s=1$, the equation in \eqref{Equation in introduction} is known in the literature as the Gross-Pitaevskii equation. It emerges when studying the behaviour of Bose-Einstein condensates up to a first-order approximation, and $\vert u\vert^2$ is the atomic density, $V$ represents an external potential, and $g$ measures the atomic interactions. We refer the reader to \cite{ESY10} and \cite{Rog13} and the references therein, for the derivation of the Gross-Pitaevskii equation for the dynamics of a Bose-Einstein condensate.

While a classical formulation of the problem often assumes regular enough potential $V$, atomic interaction coefficient $g$ and initial data $u_0$ (see, e.g. \cite{FYZ18, KOPV12}), which simplifies the theoretical analysis of the problem, real world phenomena frequently present more complex situations where the need to incorporate irregular coefficients and data is crucial, which introduces a layer of complexity to the problem that challenges the usual classical frameworks. From a physical point of view, incorporating distributional coefficients and data in physical models is a natural setting. For instance, a delta-like potential $V$ in the Gross-Pitaevskii equation can model a localised impurity in the condensate, while the distributional coefficient $g$ represents situations where the interaction strength is concentrated at a single point. Moreover, considering delta-like functions as the initial function $u_0$ can be interpreted as an initial localised excitation.

Linear and nonlinear Schr\"odinger type equations with distributional coefficients are rarely investigated in the literature. Up to our knowledge, few works consider highly singular coefficients, we cite, for instance, \cite{ARST21b, RST23, CRT22b} and more recently \cite{AACG24}. The rarity of works on this type of equation is related to the impossibility result of Schwartz \cite{Sch54}, where it was shown that distributions cannot be multiplied. This makes problems involving interactions between distributional coefficients and data impossible to be posed in any of the classical settings (strong, weak or distributional). In order to give sense to the multiplication of distributions and to provide a suitable framework where partial differential equations involving distributional data can be dealt with in a rigorous way, a new concept of solutions based on regularisation techniques was introduced in \cite{GR15} by Garetto and the second author. It was later applied to different situations. We cite for instance \cite{CRT21,CRT22a,CRT22b}, where the authors considered abstract mathematical problems, and \cite{RT17a,RT17b,Gar21,MRT19,ART19,ARST21a,ARST21b,ARST21c,SW22,ARST25} where physical models where investigated. More recently, we cite \cite{GLO21,BLO22,RSY22,RY22,CDRT23}. All these works dealt with linear equations. Our aim in this paper is to show the wide applicability of the concept of very weak solutions by considering a nonlinear equation.

\section{Preliminaries}
To start with, let us introduce some notations and notions that we will use in this paper.

\begin{itemize}
    \item By the notation $f\lesssim g$, we mean that there exists a positive constant $C$, such that $f \leq Cg$ independently on $f$ and $g$.
    \item We will be using the following version of H\"older's inequality. Let $p,q,r\in [1,\infty)$ be such that $\frac{1}{r}=\frac{1}{p} + \frac{1}{q}$. Assume that $f\in L^{p}(\mathbb{R}^d)$ and $g\in L^{q}(\mathbb{R}^d)$, then, $fg\in L^{r}(\mathbb{R}^d)$ and we have
    \begin{equation*}
        \Vert fg\Vert_{L^r}\leq \Vert f\Vert_{L^p}\Vert g\Vert_{L^q}.
    \end{equation*}
    \item Given $s>0$, the fractional Sobolev space is defined by
    \begin{equation*}
        H^s(\mathbb{R}^d) = \big\{f\in L^2(\mathbb{R}^d) : \int_{\mathbb{R}^d}(1+\vert \xi\vert^{2s})|\widehat{f}(\xi)|^2\d \xi < +\infty\big\},
    \end{equation*}
    where $\widehat{f}$ denotes the Fourier transform of $f$. We note that, the fractional Sobolev space $H^s$ endowed with the norm
    \begin{equation*}\label{Norm H^s}
        \Vert f\Vert_{H^s}:=\bigg(\int_{\mathbb{R}^d}(1+\vert \xi\vert^{2s})|\widehat{f}(\xi)|^2\d \xi\bigg)^{\frac{1}{2}},\quad\text{for } f\in H^s(\mathbb{R}^d),
    \end{equation*}
    is a Hilbert space.
    \item For $s>0$, $(-\Delta)^s$ denotes the fractional Laplacian defined by
    \begin{equation*}
    (-\Delta)^{s}f = \mathcal{F}^{-1}(\vert\xi\vert^{2s}(\widehat{f})),
    \end{equation*}
    for all $\xi\in \mathbb{R}^d$. In other words, the fractional Laplacian $(-\Delta)^s$ can be viewed as the  Fourier multiplier with symbol $\vert \xi\vert^{2s}$. With this definition and the Plancherel theorem, the fractional Sobolev space can be defined as:
    \begin{equation*}
        H^{s}(\mathbb{R}^{d})=\big\{ f\in L^{2}(\mathbb{R}^{d}): (-\Delta)^{\frac{s}{2}}f \in L^{2}(\mathbb{R}^{d})\big\},
    \end{equation*}
    moreover, the norm
    \begin{equation*}
        \Vert f\Vert_{H^{s}}:=\Vert f\Vert_{L^2}+\Vert (-\Delta)^{\frac{s}{2}}f\Vert_{L^2},
    \end{equation*}
    is equivalent to the one defined in \eqref{Norm H^s}. We refer the reader to \cite{DPV12,Gar18,Kwa17} for more details and alternative definitions.
    \item We also recall the well-known fractional Sobolev inequality. For $d\in \mathbb{N}_0$ and $s\in \mathbb{R}_+$, let $d>2s$ and $q=\frac{2d}{d-2s}$. Then, the estimate
    \begin{equation}\label{Sobolev estimate}
        \Vert f\Vert_{L^q} \leq C(d,s)\Vert (-\Delta)^{\frac{s}{2}}f\Vert_{L^2},
    \end{equation}
    holds for all $f\in H^{s}(\mathbb{R}^d)$, where the constant $C$ depends only on the dimension $d$ and the order $s$.
    \end{itemize}
    The following fractional version of the Gagliardo-Nirenberg-Sobolev inequality is frequently used throughout this paper. We first remind the reader that for $\Omega$ being a standard domain in $\mathbb{R}^d$ (i.e. $\Omega = \mathbb{R}^d$, or a Lipschitz bounded domain in $\mathbb{R}^d$), the notation $W^{s,p}(\Omega)$ refers to the usual Sobolev spaces of fractional order $s$ based on $L^p(\Omega)$ and defined as follows. For fixed $s\in (0,1)$ and any $p\in [1,+\infty)$, we have
    \begin{equation*}
        W^{s,p}(\Omega) = \Big\{ f\in L^p (\Omega): \frac{|f(x)-f(y)|}{|x-y|^{\frac{d}{p}+s}}\in L^p (\Omega \times \Omega)\Big\}.
    \end{equation*}
    When endowed with the natural norm
    \begin{equation*}
        \Vert f \Vert_{W^{s,p}(\Omega)}:= \Big(\Vert f \Vert_{L^p (\Omega)}^p + [f]_{W^{s,p}(\Omega)}^p\Big)^{\frac{1}{p}},
    \end{equation*}
    where
    \begin{equation*}
        [f]_{W^{s,p}(\Omega)} := \bigg(\int_{\Omega} \int_{\Omega} \bigg( \frac{|f(x)-f(y)|}{|x-y|^{\frac{d}{p}+s}} \bigg)^p\d x\d y\bigg)^{\frac{1}{p}},
    \end{equation*}
    is the so-called Gagliardo norm, $W^{s,p}(\Omega)$ is a Banach space.

    In the case when $s>1$ and it is not an integer, we write $s=m+\sigma$, where $m$ is integer and $\sigma \in (0,1)$. Here, $W^{s,p}(\Omega)$ is defined by
    \begin{equation*}
        W^{s,p}(\Omega):= \big\{ f\in W^{m,p}(\Omega): \partial^{\alpha}f \in W^{\sigma,p}(\Omega),~\text{for any}~\alpha : |\alpha|=m \big\},
    \end{equation*}
    where $|\alpha|$ denotes the length of the multi-index $\alpha$. This space is a Banach space with respect to the norm
    \begin{equation*}
        \Vert f \Vert_{W^{s,p}(\Omega)}:= \Big( \Vert f \Vert^p_{W^{m,p}(\Omega)} + \sum_{|\alpha|=m}\Vert D^{\alpha}f \Vert_{W^{\sigma,p}(\Omega)}^p \Big)^{\frac{1}{p}}.
    \end{equation*}
    We refer the reader to \cite{DPV12} for more details.
    
\begin{prop}[\textbf{Fractional GNS inequality}, e.g. Theorem 1. \cite{BM19}]\label{Proposition GNS inequality}
    Let $\Omega$ be a standard domain in $\mathbb{R}^d$. Then, for $r,s_1,s_2,p_1,p_2,q,\theta,d$ satisfying
    \begin{equation*}
        0\leq s_1\leq s_2,~r\geq 0,~ 1\leq p_1,p_2,q\leq \infty,~(s_1,p_1)\neq(s_2,p_2),~\theta \in (0,1),
    \end{equation*}
    \begin{equation*}
        r<\mu:= \theta s_1 + (1-\theta)s_2,~\frac{1}{q}=\bigg(\frac{\theta}{p_1} + \frac{1-\theta}{p_2}\bigg) - \frac{\mu-r}{d},
    \end{equation*}
    the inequality
    \begin{equation}\label{GNS inequality}
        \big\Vert f\big\Vert_{W^{r,q}(\Omega)} \lesssim \big\Vert f\big\Vert_{W^{s_1,p_1}(\Omega)}^{\theta} \big\Vert f\big\Vert_{W^{s_2,p_2}(\Omega)}^{1-\theta},
    \end{equation}
    holds for all $f\in W^{s_1,p_1}(\Omega) \cap W^{s_2,p_2}(\Omega)$, with the following exceptions, when it fails:
    \begin{enumerate}
        \item[1.] $d=1$, $s_2$ is an integer $\geq 1$, $1<p_1\leq \infty$, $p_2 =1$, $s_1 = s_2 -1+\frac{1}{p_1}$,\\
        $[1<p_1<\infty],~r=s_2 -1$ or $\Big[ s_2 +\frac{\theta}{p_1}-1<r<s_2 +\frac{\theta}{p_1}-\theta \Big]$,
        \item[2.] $d\geq 1,~ s_1<s_2,~ s_1 -\frac{d}{p_1} = s_2 -\frac{d}{p_2}=r$ is an integer, $q=\infty$, $(p_1,p_2)\neq (\infty,1)$ (for every $\theta\in (0,1)$).
    \end{enumerate}
\end{prop}

\subsection{Duhamel's principle}
We prove the following special version of Duhamel's principle that will frequently be used throughout this paper. For more general versions of this principle,  we refer the reader to e.g. \cite{ER18}. Let us consider the following non-homogeneous problem for the first order linear evolution equation
\begin{equation}
    \left\lbrace
    \begin{array}{l}
    u_t(t,x)-Lu(t,x)=f(t,x) ,~~~(t,x)\in\left(0,\infty\right)\times \mathbb{R}^{d},\\
    u(0,x)=u_{0}(x),\,\,\, x\in\mathbb{R}^{d}, \label{Equation Duhamel 2}
    \end{array}
    \right.
\end{equation}
where $L$ is a linear partial differential operator that includes no time derivatives.

\begin{prop}\label{Prop. Duhamel 2}
The solution to the Cauchy problem \eqref{Equation Duhamel 2} is given by
\begin{equation}
    u(t,x)= w(t,x) + \int_0^t v(t,x;s)\d s,\label{Sol Duhamel 2}
\end{equation}
where $w(t,x)$ is the solution to the homogeneous problem
\begin{equation}
    \left\lbrace
    \begin{array}{l}
    w_t(t,x)-Lw(t,x)=0 ,~~~(t,x)\in\left(0,\infty\right)\times \mathbb{R}^{d},\\
    w(0,x)=u_{0}(x),\,\,\, x\in\mathbb{R}^{d},\label{Homog eqn Duhamel 2}
    \end{array}
    \right.
\end{equation}
and $v(t,x;s)$ solves the auxiliary Cauchy problem
\begin{equation}
    \left\lbrace
    \begin{array}{l}
    v_t(t,x;s)-Lv(t,x;s)=0 ,~~~(t,x)\in\left(s,\infty\right)\times \mathbb{R}^{d},\\
    v(s,x;s)=f(s,x),\,\,\, x\in\mathbb{R}^{d},\label{Aux eqn Duhamel 2}
    \end{array}
    \right.
\end{equation}
where $s$ is a time-like parameter varying over $\left(0,\infty\right)$.
\end{prop}

\begin{proof}
Differentiating \eqref{Sol Duhamel 2} with respect to the variable $t$, gives
\begin{equation}
    \partial_t u(t,x)=\partial_t w(t,x) + v(t,x;t) + \int_0^t \partial_t v(t,x;s)\d s.\label{Duhamel proof 2.1}
\end{equation}
We note that $v(t,x;t)=f(t,x)$ by the initial condition in \eqref{Aux eqn Duhamel 2}.
Since the differential operator $L$ is acting only on the variable $x$, by applying $L$ in \eqref{Sol Duhamel 2} we get 
\begin{equation}
    L u(t,x)=L w(t,x) + \int_0^t L v(t,x;s)\d s.\label{Duhamel proof 2.2}
\end{equation}
Combining \eqref{Duhamel proof 2.1} and (\ref{Duhamel proof 2.2}) and using that $w$ and $v$ are the solutions to \eqref{Homog eqn Duhamel 2} and \eqref{Aux eqn Duhamel 2} we obtain
\begin{equation*}
    u_t(t,x) - Lu(t,x)= f(t,x).
\end{equation*}
Observing that $u(0,x)=u_0(x)$ from the initial condition in \eqref{Homog eqn Duhamel 2} completes the proof.
\end{proof}

\subsection{Gronwall's type inequalities} The following integral form of Gronwall's inequality (also referred to as the Bellman inequality) is frequently used in this paper. For the proof, we refer the reader to e.g. \cite[Theorem 1.1.2.]{Qin16}.

\begin{lem}\label{Lemma: Gronwall}
    Let $X(t)$ and $a(t)$ be non-negative, continuous functions on $[0,T]$, such that for all $t\in [0,T]$,
    \begin{equation*}
        X(t) \leq \lambda + \int_{0}^{t} a(s)X(s)\d s,
    \end{equation*}
    where $\lambda$ is a non-negative constant. Then
    \begin{equation*}
        X(t) \leq \lambda \exp{\Big(\int_{0}^{t} a(s)\d s\Big)},
    \end{equation*}
    for all $t\in [0,T]$.
\end{lem}

\subsection{Energy estimates for the classical solution}
The following lemma is essential in proving our main results. It is stated when the equation coefficients and Cauchy data are regular enough for a classical solution to exist.

\begin{lem}\label{Lemma1}
Fix $s>0$. Let $g,V\in L^{\infty}(\mathbb{R}^d)$ be non-negative and suppose that $u_0 \in H^{s}(\mathbb{R}^d)\cap L^4(\mathbb{R}^d)$. Then there is a solution $u\in C([0,T];H^{s}(\mathbb{R}^d))$ to the Cauchy problem \eqref{Equation in introduction} that satisfies the estimates
    \begin{equation}\label{Estimate0_lemma1}
        \Vert u(t,\cdot)\Vert_{L^2} = \Vert u_0\Vert_{L^2},
    \end{equation}
    \begin{equation}\label{Estimate1_lemma1}
        \Vert u(t,\cdot)\Vert_{H^s} \lesssim \left( 1+\Vert V\Vert_{L^{\infty}}\right)^{\frac{1}{2}}\Vert u_0\Vert_{H^s} + \Vert g\Vert_{L^{\infty}}^{\frac{1}{2}}  \Vert u_0\Vert_{L^4}^{2},
    \end{equation}
    \begin{equation}
    \Vert V^{\frac{1}{2}}(\cdot)u(t,\cdot)\Vert_{L^2} \lesssim \left( 1+\Vert V\Vert_{L^{\infty}}\right)^{\frac{1}{2}}\Vert u_0\Vert_{H^s} + \Vert g\Vert_{L^{\infty}}^{\frac{1}{2}}  \Vert u_0\Vert_{L^4}^{2},
\end{equation}
and
\begin{equation}
    \Vert g^{\frac{1}{4}}(\cdot)u(t,\cdot)\Vert_{L^4} \lesssim \left[\left( 1+\Vert V\Vert_{L^{\infty}}\right)^{\frac{1}{2}}\Vert u_0\Vert_{H^s} + \Vert g\Vert_{L^{\infty}}^{\frac{1}{2}}  \Vert u_0\Vert_{L^4}^{2}\right]^{\frac{1}{2}},
\end{equation}
    for all $t\in [0,T]$. Moreover, if $2s<d<4s$, then the unique solution to \eqref{Equation in introduction} satisfies the estimates
    \begin{equation}\label{Estimate2_lemma1}
        \Vert u(t,\cdot)\Vert_{H^s} \lesssim \left( 1+\Vert V\Vert_{L^{\frac{d}{2s}}}\right)^{\frac{1}{2}}\Vert u_0\Vert_{H^s} + \Vert g\Vert_{L^{\frac{d}{4s-d}}}^{\frac{1}{2}}  \Vert u_0\Vert_{H^s}^{2},
    \end{equation}
    \begin{equation}
    \Vert V^{\frac{1}{2}}(\cdot)u(t,\cdot)\Vert_{L^2} \lesssim \left( 1+\Vert V\Vert_{L^{\frac{d}{2s}}}\right)^{\frac{1}{2}}\Vert u_0\Vert_{H^s} + \Vert g\Vert_{L^{\frac{d}{4s-d}}}^{\frac{1}{2}}  \Vert u_0\Vert_{H^s}^{2},
\end{equation}
and
\begin{equation}
    \Vert g^{\frac{1}{4}}(\cdot)u(t,\cdot)\Vert_{L^4} \lesssim \left[\left( 1+\Vert V\Vert_{L^{\frac{d}{2s}}}\right)^{\frac{1}{2}}\Vert u_0\Vert_{H^s} + \Vert g\Vert_{L^{\frac{d}{4s-d}}}^{\frac{1}{2}}  \Vert u_0\Vert_{H^s}^{2}\right]^{\frac{1}{2}},
\end{equation}
for all $t\in [0,T]$.
\end{lem}

\begin{proof}
    We multiply the equation in \eqref{Equation in introduction} by $-i$ to get
    \begin{equation*}
        u_{t}(t,x) + i(-\Delta)^{s} u(t,x) + iV(x)u(t,x) + ig(x)\vert u(t,x)\vert^2 u(t,x)=0.
    \end{equation*}
    Multiplying the obtained equation by $\overline{u}$ and integrating with respect to the variable $x$ over $\mathbb{R}^d$ and taking the real part, we easily see that
    \begin{equation*}
        \Re\langle u_t(t,\cdot),u(t,\cdot)\rangle_{L^2} = \frac{1}{2}\partial_{t}\langle u(t,\cdot),u(t,\cdot)\rangle_{L^2} = \frac{1}{2}\partial_{t}\Vert u(t,\cdot)\Vert_{L^2}^2,
    \end{equation*}
    and
    \begin{equation*}
        \Re\Big( i\langle (-\Delta)^{s} u(t,\cdot),u(t,\cdot)\rangle_{L^2} + i\langle V(\cdot) u(t,\cdot),u(t,\cdot)\rangle_{L^2} + i\langle g(\cdot)\vert u(t,\cdot)\vert^2  u(t,\cdot),u(t,\cdot)\rangle_{L^2} \Big) = 0.
    \end{equation*}
    The first conservation law for the system follows and we get
    \begin{equation}\label{Estim0_proof1}
        \Vert u(t,\cdot)\Vert_{L^2} = \Vert u_0\Vert_{L^2},
    \end{equation}
    for all $t\in [0,T]$. Let us derive the second conservation law. For this, we multiply the equation in (\ref{Equation in introduction}) by $\overline{u_t}$, integrate over $\mathbb{R}^d$ and take the real part, to get
\begin{align*}
    \Re\langle i\partial_{t}u(t,\cdot),\partial_{t}u(t,\cdot)\rangle_{L^2} = & \Re\Big(\langle(-\Delta)^{s}u(t,\cdot),\partial_{t}u(t,\cdot)\rangle_{L^2} +  \langle V(\cdot)u(t,\cdot),\partial_{t}u(t,\cdot)\rangle_{L^2} \\
    & + \langle g(\cdot)\vert u(t,\cdot)\vert^2 u(t,\cdot),\partial_{t}u(t,\cdot)\rangle_{L^2} \Big).
\end{align*}
We have
\begin{equation*}
    \Re\langle i\partial_{t}u(t,\cdot),\partial_{t}u(t,\cdot)\rangle_{L^2} = 0,
\end{equation*}
\begin{equation*}
    \Re\langle(-\Delta)^{s}u(t,\cdot),\partial_{t}u(t,\cdot)\rangle_{L^2} = \frac{1}{2}\partial_{t}\Vert (-\Delta)^{\frac{s}{2}}u(t,\cdot)\Vert_{L^2}^{2},
\end{equation*}
\begin{equation*}
    \Re\langle V(\cdot)u(t,\cdot),\partial_{t}u(t,\cdot)\rangle_{L^2} = \frac{1}{2}\partial_{t}\Vert V^{\frac{1}{2}}(\cdot)u(t,\cdot)\Vert_{L^2}^{2},
\end{equation*}
and
\begin{equation*}
    \Re\langle g(\cdot)\vert u(t,\cdot)\vert^2 u(t,\cdot),\partial_{t}u(t,\cdot)\rangle_{L^2} = \frac{1}{4}\partial_{t}\Vert g^{\frac{1}{4}}(\cdot)u(t,\cdot)\Vert_{L^4}^{4}.
\end{equation*}
The last identity comes from the fact that for all $(t,x)\in [0,T]\times \mathbb{R}^d$, we have
\begin{equation*}
    \Re\left(\vert u(t,x)\vert^2 u(t,x)\overline{u}_{t}(t,x)\right) = \frac{1}{4}\partial_{t}\vert u(t,x)\vert^{4}.
\end{equation*}
Then, the Hamiltonian of the system
\begin{equation*}
    H(t) := \Vert (-\Delta)^{\frac{s}{2}}u(t,\cdot)\Vert_{L^2}^{2} + \Vert V^{\frac{1}{2}}(\cdot)u(t,\cdot)\Vert_{L^2}^{2} + \frac{1}{2}\Vert g^{\frac{1}{4}}(\cdot)u(t,\cdot)\Vert_{L^4}^{4},
\end{equation*}
is conserved in time, that is
\begin{equation*}
    H(t) = H(0),
\end{equation*}
for all $t\in[0,T]$. By the remark that $\Vert V^{\frac{1}{2}}(\cdot)u(t,\cdot)\Vert_{L^2}^{2}$ and $\Vert g^{\frac{1}{4}}(\cdot)u(t,\cdot)\Vert_{L^4}^{4}$ can be estimated by
\begin{equation*}
    \Vert V^{\frac{1}{2}}(\cdot)u_0(\cdot)\Vert_{L^2}^{2} \leq \Vert V\Vert_{L^{\infty}}  \Vert u_0\Vert_{L^2}^{2},
\end{equation*}
and
\begin{equation*}
    \Vert g^{\frac{1}{4}}(\cdot)u_0(\cdot)\Vert_{L^4}^{4} \leq \Vert g\Vert_{L^{\infty}}  \Vert u_0\Vert_{L^4}^{4},
\end{equation*}
respectively, we get
\begin{equation}\label{Estim1_proof1}
    \Vert (-\Delta)^{\frac{s}{2}}u(t,\cdot)\Vert_{L^2} \lesssim \left( 1+\Vert V\Vert_{L^{\infty}}\right)^{\frac{1}{2}}\Vert u_0\Vert_{H^s} + \Vert g\Vert_{L^{\infty}}^{\frac{1}{2}}  \Vert u_0\Vert_{L^4}^{2},
\end{equation}
\begin{equation}\label{Estim2_proof1}
    \Vert V^{\frac{1}{2}}(\cdot)u(t,\cdot)\Vert_{L^2} \lesssim \left( 1+\Vert V\Vert_{L^{\infty}}\right)^{\frac{1}{2}}\Vert u_0\Vert_{H^s} + \Vert g\Vert_{L^{\infty}}^{\frac{1}{2}}  \Vert u_0\Vert_{L^4}^{2},
\end{equation}
and
\begin{equation}\label{Estim3_proof1}
    \Vert g^{\frac{1}{4}}(\cdot)u(t,\cdot)\Vert_{L^4} \lesssim \left[\left( 1+\Vert V\Vert_{L^{\infty}}\right)^{\frac{1}{2}}\Vert u_0\Vert_{H^s} + \Vert g\Vert_{L^{\infty}}^{\frac{1}{2}}  \Vert u_0\Vert_{L^4}^{2}\right]^{\frac{1}{2}},
\end{equation}
proving the first part of the lemma. Now, let $2s<d<4s$. We repeat the reasoning above, arriving at the conservation of the Hamiltonian
\begin{equation*}
    H(t) := \Vert (-\Delta)^{\frac{s}{2}}u(t,\cdot)\Vert_{L^2}^{2} + \Vert V^{\frac{1}{2}}(\cdot)u(t,\cdot)\Vert_{L^2}^{2} + \frac{1}{2}\Vert g^{\frac{1}{4}}(\cdot)u(t,\cdot)\Vert_{L^4}^{4},
\end{equation*}
that is
\begin{equation}\label{Conserv. Hamiltonian}
    H(t)= \Vert (-\Delta)^{\frac{s}{2}}u_0\Vert_{L^2}^{2} + \Vert V^{\frac{1}{2}}(\cdot)u_0\Vert_{L^2}^{2} + \frac{1}{2}\Vert g^{\frac{1}{4}}(\cdot)u_0\Vert_{L^4}^{4},
\end{equation}
for all $t\in [0,T]$. By applying H\"older's inequality, the second and third terms in \eqref{Conserv. Hamiltonian} are estimated by
\begin{equation*}
    \Vert V^{\frac{1}{2}}(\cdot)u_0\Vert_{L^2} \leq \Vert V^{\frac{1}{2}}\Vert_{L^p} \Vert u_0\Vert_{L^q} \leq \Vert V\Vert^{\frac{1}{2}}_{L^{\frac{p}{2}}} \Vert u_0\Vert_{L^q},
\end{equation*}
and
\begin{equation*}
    \Vert g^{\frac{1}{4}}(\cdot)u_0\Vert_{L^4} \leq \Vert g^{\frac{1}{4}}\Vert_{L^{p'}} \Vert u_0\Vert_{L^{q'}} \leq \Vert g\Vert^{\frac{1}{4}}_{L^{\frac{p'}{4}}} \Vert u_0\Vert_{L^{q'}},
\end{equation*}
with $p,q,p',q'\in (0,\infty)$, such that $\frac{1}{2}=\frac{1}{p} + \frac{1}{q}$ and $\frac{1}{4}=\frac{1}{p'} + \frac{1}{q'}$. Now, if we choose $q=q' = \frac{2d}{d-2s}$ and consequently $p = \frac{d}{s}$ and $p'=\frac{4d}{4s-d}$, then, Sobolev inequality applies and we get
\begin{equation}\label{Estim4_proof1}
    \Vert V^{\frac{1}{2}}(\cdot)u_0\Vert_{L^2} \lesssim \Vert V\Vert^{\frac{1}{2}}_{L^{\frac{d}{2s}}} \Vert(-\Delta)^{\frac{s}{2}} u_0\Vert_{L^2} \leq \Vert V\Vert^{\frac{1}{2}}_{L^{\frac{d}{2s}}} \Vert u_0\Vert_{H^s},
\end{equation}
and
\begin{equation}\label{Estim5_proof1}
    \Vert g^{\frac{1}{4}}(\cdot)u_0\Vert_{L^4} \lesssim \Vert g\Vert^{\frac{1}{4}}_{L^{\frac{d}{4s-d}}} \Vert(-\Delta)^{\frac{s}{2}} u_0\Vert_{L^2} \leq \Vert g\Vert^{\frac{1}{4}}_{L^{\frac{d}{4s-d}}} \Vert u_0\Vert_{H^s}.
\end{equation}
The desired estimates follow by substituting \eqref{Estim4_proof1} and \eqref{Estim5_proof1} into \eqref{Conserv. Hamiltonian}, ending the proof.
\end{proof}

\section{Very weak well-posedness}
In this section, we consider the irregular case when the equation coefficients and initial data are distributions, and we prove that the Cauchy problem
\begin{equation}\label{Equation in VW well-posedness}
    \left\lbrace
    \begin{array}{l}
    iu_{t}(t,x) = (-\Delta)^{s} u(t,x) + V(x)u(t,x) + g(x)\vert u(t,x)\vert^2 u(t,x) ,~~~(t,x)\in\left[0,T\right]\times \mathbb{R}^{d},\\
    u(0,x)=u_{0}(x),
    \end{array}
    \right.
\end{equation}
is well-posed in a very weak sense. To start with, we regularise $V$, $g$, and the initial data $u_0$ by convolution with a mollifying net 
\begin{equation*}
    \psi_\varepsilon (x)=\omega(\varepsilon)^{-d}\psi\left(x/\omega(\varepsilon)\right),
\end{equation*}
where $\omega(\varepsilon)$ is a positive function converging to $0$ as $\varepsilon\rightarrow 0$ to be chosen later, and $\psi$ is a Friedrichs mollifier, that is $\psi\in C_0^{\infty}(\mathbb{R}^d)$ is non-negative and $\int_{\mathbb{R}^d}\psi(x)\d x=1$. This yields families (nets) of smooth functions $(V_\varepsilon)_{\varepsilon\in(0,1]}=(V\ast \psi_\varepsilon)_{\varepsilon\in(0,1]}$, $(g_\varepsilon)_{\varepsilon\in(0,1]}=(g\ast \psi_\varepsilon)_{\varepsilon\in(0,1]}$, and $(u_{0,\varepsilon})_{\varepsilon\in(0,1]}=(u_{0}\ast \psi_\varepsilon)_{\varepsilon\in(0,1]}$, and accordingly a net of regularised problems, depending on the regularising parameter $\varepsilon \in (0,1]$. That is:
    \begin{equation}\label{Regularised problems}
        \bigg\{
        \begin{array}{l}
        i\partial_{t}u_{\varepsilon}(t,x) = (-\Delta)^{s}u_{\varepsilon}(t,x) + V_{\varepsilon}(x)u_{\varepsilon}(t,x) + g_{\varepsilon}(x)\vert u_{\varepsilon}(t,x)\vert^2 u_{\varepsilon}(t,x)=0,\\
        u_{\varepsilon}(0,x)=u_{0,\varepsilon}(x),
        \end{array}
    \end{equation}
with $(t,x)\in [0,T]\times\mathbb{R}^d$.

The following definitions are needed for proving the existence and uniqueness of very weak solutions to \eqref{Equation in VW well-posedness}. Let $f$ be a function (or distribution), and let $(f_\varepsilon)_{\varepsilon\in(0,1]} = (f\ast \psi_\varepsilon)_{\varepsilon\in(0,1]}$ be the regularisation of $f$ obtained by convolution with a mollifying net $\left(\psi_\varepsilon (x)\right)_{\varepsilon\in (0,1]}=\left(\omega(\varepsilon)^{-d}\psi\left(x/\omega(\varepsilon)\right)\right)_{\varepsilon\in (0,1]}$.

\begin{defn}[\textbf{Moderateness}]\label{Defn moderateness}
Let $X,\Vert \cdot\Vert_X$ be a normed vector space of functions on $\mathbb{R}^d$.
\begin{enumerate}
    \item[\textsc{1.}] A net of functions $(f_\varepsilon)_{\varepsilon\in(0,1]}$ from $X$ is said to be $X$-moderate, if there exist $N\in\mathbb{N}_0$ such that
\begin{equation*}\label{Defn moderateness1}
    \Vert f_\varepsilon\Vert_X \lesssim \omega(\varepsilon)^{-N}.
\end{equation*}
    \item[\textsc{2.}] A net of functions $(f_\varepsilon)_{\varepsilon\in(0,1]}$ from $X\cap Y$ is said to be $X\cap Y$-moderate, if there exist $N\in\mathbb{N}^\ast$ such that
\begin{equation*}
    \max \big\{\Vert f_\varepsilon\Vert_X , \Vert f_\varepsilon\Vert_Y \big\} \lesssim \omega(\varepsilon)^{-N}.\label{Eqn def moderateness2}
\end{equation*}
    \item[\textsc{3.}] For $T>0$, a net of functions $(u_\varepsilon(\cdot,\cdot))_{\varepsilon\in(0,1]}$ from $C\big([0,T];H^{s}(\mathbb{R}^d)\big)$ is said to be $C\big([0,T];H^{s}(\mathbb{R}^d)\big)$-moderate, if there exist $N\in\mathbb{N}_0$ such that
    \begin{equation*}\label{Defn moderateness3}
        \sup_{t\in [0,T]}\Vert u_\varepsilon(t,\cdot)\Vert_{H^s} \lesssim \omega(\varepsilon)^{-N}.
    \end{equation*}
\end{enumerate}
For the third definition of moderateness, we will shortly write $C$-moderate instead of $C\big([0,T];H^{s}(\mathbb{R}^d)\big)$-moderate.
\end{defn}

\begin{rem}
    We should highlight here that moderateness is a natural assumption for distributions. Indeed, if for instance, we consider the Dirac delta function or its powers (we understand powers of delta function not as distributions, but in the sense of their representatives at the level of regularisation), then by regularisation with a mollifying net, for $k\in \mathbb{N}_0$, we get:
    \begin{equation*}
        \delta^k \ast \psi_{\varepsilon}(x) = \varepsilon^{-k d}\psi^{k}(\varepsilon^{-1}x) \leq C\varepsilon^{-k d},
    \end{equation*}
    for $\omega(\varepsilon)=\varepsilon$.
\end{rem}

\subsection{Existence of very weak solutions}
We are now in position to introduce the notion of very weak solution adapted to our problem.

\begin{defn}[\textbf{Very weak solution}]\label{Defn1 V.W.S}
    A net of functions $(u_{\varepsilon})_{\varepsilon}\in C([0,T];H^{s}(\mathbb{R}^d))$ is said to be a very weak solution to the Cauchy problem (\ref{Equation in VW well-posedness}), if there exist
    \begin{itemize}
        \item $L^{\infty}(\mathbb{R}^d)$-moderate regularisations $(V_{\varepsilon})_{\varepsilon}$ and $(g_{\varepsilon})_{\varepsilon}$ to $V$ and $g$, with $V_{\varepsilon} \geq 0$ and $g_{\varepsilon} \geq 0$,
        \item $H^{s}(\mathbb{R}^d)\cap L^4(\mathbb{R}^d)$-moderate regularisation $(u_{0,\varepsilon})_{\varepsilon}$ to $u_0$,
    \end{itemize}
    such that, $(u_{\varepsilon})_{\varepsilon}$ solving the regularised problems \eqref{Regularised problems} for every $\varepsilon\in (0,1]$, is $C$-moderate.
\end{defn}

\begin{rem}
    Roughly speaking, a very weak solution to \eqref{Equation in VW well-posedness} is the family of weak solutions to \eqref{Regularised problems}, satisfying a uniform moderateness condition.
\end{rem}

We are all set to state our existence result for which the proof is straightforward.

\begin{thm}\label{Existence theorem}
    Assume that there exists $L^{\infty}(\mathbb R^d)$-moderate regularisations to $V$ and $g$, with $V_{\varepsilon} \geq 0$ and $g_{\varepsilon} \geq 0$, and $H^{s}(\mathbb{R}^d)\cap L^4(\mathbb{R}^d)$-moderate regularisation to $u_0$. Then, the Cauchy problem \eqref{Equation in VW well-posedness} has a very weak solution.
\end{thm}
\begin{rem}
    The assumptions on the existence of such $V_{\varepsilon}, g_{\varepsilon}$ and $u_{0,\varepsilon}$ are satisfied when, for example, $V, g, u_0 \in \mathcal{E}'(\mathbb{R}^d),$ with $V\ge 0$ and $g\ge 0$ in the sense of distributions.
\end{rem}
\begin{proof}
    Let $V$, $g$ and $u_0$ satisfy the moderateness assumptions for their regularising families. Then, there exist $N_1 , N_2 , N_3 \in \mathbb{N}$, such that
    \begin{equation*}
        \Vert V_{\varepsilon}\Vert_{L^\infty} \lesssim \omega(\varepsilon)^{-N_1},~~\Vert g_{\varepsilon}\Vert_{L^\infty} \lesssim \omega(\varepsilon)^{-N_2},
    \end{equation*}
    and
    \begin{equation*}
        \Vert u_{0,_{\varepsilon}}\Vert_{H^{s}\cap L^4} \lesssim \omega(\varepsilon)^{-N_3}.
    \end{equation*}
    It follows from \eqref{Estimate1_lemma1} that
    \begin{equation*}
        \Vert u_{\varepsilon}(t,\cdot)\Vert_{H^s} \lesssim \omega(\varepsilon)^{-\frac{\max\{N_1 ,N_2\}}{2}-2N_3},
    \end{equation*}
    uniformly in $t\in [0,T]$. The net $(u_{\varepsilon})_{\varepsilon}$ is then $C$-moderate and the existence of a very weak solution follows, ending the proof.
\end{proof}

\subsection{Uniqueness}
In all what follows, we will only consider the case when $d<2s$. When the equation coefficients $V$ and $g$ belong to $L^{\infty}(\mathbb{R}^d)$, we prove uniqueness of the very weak solution. Before stating our result, we need the following definition:

\begin{defn}[\textbf{Negligibility}]\label{Defn negligibility}
    Let $ \left(X,\Vert\cdot\Vert_X\right) $ be a normed vector space. A net of functions $(f_\varepsilon)_{\varepsilon\in(0,1]}$ from $X$ is said to be $X$-negligible, if the estimate
    \begin{equation}\label{Negligibility formula}
        \Vert f_\varepsilon\Vert_X \lesssim \varepsilon^k,
    \end{equation}
    is valid for any $k>0$ and $\varepsilon \in (0,1]$.
\end{defn}

We prove uniqueness of the very weak solution in the following sense:
\begin{defn}[\textbf{Uniqueness}]\label{Defn1 uniqueness}
    We say that the Cauchy problem \eqref{Equation in VW well-posedness}, has a unique very weak solution, if for all $L^{\infty}$ nets $(V_{\varepsilon})_{\varepsilon}$, $(\Tilde{V}_{\varepsilon})_{\varepsilon}$ and $(g_{\varepsilon})_{\varepsilon}$, $(\Tilde{g}_{\varepsilon})_{\varepsilon}$ for the equation coefficients $V$ and $g$, and for any $L^2$ nets $(u_{0,\varepsilon})_{\varepsilon}$, $(\Tilde{u}_{0,\varepsilon})_{\varepsilon}$ for the initial function $u_0$, such that the nets $(V_{\varepsilon}-\Tilde{V}_{\varepsilon})_{\varepsilon}$, $(g_{\varepsilon}-\Tilde{g}_{\varepsilon})_{\varepsilon}$, and $(u_{0,\varepsilon}-\Tilde{u}_{0,\varepsilon})_{\varepsilon}$ are $\big\{L^{\infty}(\mathbb{R}^d),L^{\infty}(\mathbb{R}^d),L^{2}(\mathbb{R}^d)\big\}$-negligible, it follows that the net
    \begin{equation*}
        \big(u_{\varepsilon}(t,\cdot)-\Tilde{u}_{\varepsilon}(t,\cdot)\big)_{\varepsilon}
    \end{equation*}
    is $L^2(\mathbb{R}^d)$-negligible for all $t\in [0,T]$, where $(u_{\varepsilon})_{\varepsilon}$ and $(\Tilde{u}_{\varepsilon})_{\varepsilon}$ are the nets of solutions to the regularised problems
    \begin{equation}\label{Regularised problems for u}
        \bigg\{
        \begin{array}{l}
        i\partial_{t}u_{\varepsilon}(t,x) = (-\Delta)^{s}u_{\varepsilon}(t,x) + V_{\varepsilon}(x)u_{\varepsilon}(t,x) + g_{\varepsilon}(x)\vert u_{\varepsilon}(t,x)\vert^2 u_{\varepsilon}(t,x),\\
        u_{\varepsilon}(0,x)=u_{0,\varepsilon}(x),
        \end{array}
    \end{equation}
    and
    \begin{equation}\label{Regularised problems for u tilde}
        \bigg\{
        \begin{array}{l}
        i\partial_{t}\tilde{u}_{\varepsilon}(t,x) = (-\Delta)^{s}\tilde{u}_{\varepsilon}(t,x) + \tilde{V}_{\varepsilon}(x)\tilde{u}_{\varepsilon}(t,x) + \tilde{g}_{\varepsilon}(x)\vert \tilde{u}_{\varepsilon}(t,x)\vert^2 \tilde{u}_{\varepsilon}(t,x),\\
        \tilde{u}_{\varepsilon}(0,x)=\tilde{u}_{0,\varepsilon}(x),
        \end{array}
    \end{equation}
    for $(t,x) \in [0,T]\times\mathbb{R}^d$, respectively.
\end{defn}

\begin{rem}
    The uniqueness will be shown in the sense that very weak solutions are independent of negligibly different regularisations. In other words, in the sense that negligible changes in the regularisations of the equation coefficients and initial data, lead to negligible changes in the corresponding very weak solutions.
\end{rem}

\begin{thm}\label{Uniqueness theorem}
    Let $T>0$. Let $d\ge 1$, and let $s>\frac{d}{2}$. Assume that $V, g \in \mathcal{E}'(\mathbb{R}^d)$ and $V,g \geq 0$. Under the assumptions of Theorem \ref{Existence theorem}, the very weak solution to the Cauchy problem \eqref{Equation in VW well-posedness} is unique.
\end{thm}

\begin{proof}
    For $T>0$, let $(u_{\varepsilon})_{\varepsilon}$ and $(\Tilde{u}_{\varepsilon})_{\varepsilon}$ be the nets of solutions to \eqref{Regularised problems for u} and \eqref{Regularised problems for u tilde} respectively, and assume that $(V_{\varepsilon}-\Tilde{V}_{\varepsilon})_{\varepsilon}$, $(g_{\varepsilon}-\Tilde{g}_{\varepsilon})_{\varepsilon}$, and $(u_{0,\varepsilon}-\Tilde{u}_{0,\varepsilon})_{\varepsilon}$ are $\big\{L^{\infty}(\mathbb{R}^d),L^{\infty}(\mathbb{R}^d),L^{2}(\mathbb{R}^d)\big\}$-negligible. Let us denote by
    \begin{equation*}
        U_{\varepsilon}(t,x):=u_{\varepsilon}(t,x)-\Tilde{u}_{\varepsilon}(t,x).
    \end{equation*}
    Then, $U$ is a solution to
    \begin{equation}\label{Equation U in uniqueness}
        \bigg\{
        \begin{array}{l}
        i\partial_{t}U_{\varepsilon}(t,x) - (-\Delta)^{s}U_{\varepsilon}(t,x) - V_{\varepsilon}(x)U_{\varepsilon}(t,x) = f_\varepsilon (t,x),~~~(t,x)\in\left[0,T\right]\times \mathbb{R}^{d},\\
        U_{\varepsilon}(0,x)=(u_{0,\varepsilon}-\Tilde{u}_{0,\varepsilon})(x),~~~x\in \mathbb{R}^{d},
        \end{array}
    \end{equation}
    where
    \begin{align*}
        f_\varepsilon (t,x):= & \big( V_{\varepsilon}(x)-\tilde{V}_{\varepsilon}(x)\big)\tilde{u}_{\varepsilon}(t,x) + \big( g_{\varepsilon}(x)-\tilde{g}_{\varepsilon}(x) \big)|\tilde{u}_{\varepsilon}(t,x)|^2 \tilde{u}_{\varepsilon}(t,x)\\
        & + g_{\varepsilon}(x)\Big( |u_{\varepsilon}(t,x)|^2 u_{\varepsilon}(t,x) - |\tilde{u}_{\varepsilon}(t,x)|^2 \tilde{u}_{\varepsilon}(t,x) \Big).
    \end{align*}
    It is easy to see that the homogeneous problem associated to \eqref{Equation U in uniqueness} satisfies the first conservation law derived in Lemma \ref{Lemma1}. Our strategy to handle the nonlinear terms appearing in \eqref{Equation U in uniqueness}, is to take all these terms to the right hand side and to use energy estimates from Lemma \ref{Lemma1}. Thus, by considering $f$ as a source term, and applying Duhamel's principle (see Proposition \ref{Prop. Duhamel 2}), the solution $U$ to \eqref{Equation U in uniqueness} has the representation
    \begin{equation}\label{Duhamel represent. U in uniqueness}
        U_{\varepsilon}(t,x) = W_{\varepsilon}(t,x) + \int_{0}^{t}Y_{\varepsilon}(t,x;\tau)\d \tau,
    \end{equation}
    where $W_{\varepsilon}$ is the solution to the homogeneous problem
    \begin{equation}\label{Equation W in uniqueness}
        \bigg\{
        \begin{array}{l}
        i\partial_{t}W_{\varepsilon}(t,x) - (-\Delta)^{s}W_{\varepsilon}(t,x) - V_{\varepsilon}(x)W_{\varepsilon}(t,x) = 0,~~~(t,x)\in\left[0,T\right]\times \mathbb{R}^{d},\\
        W_{\varepsilon}(0,x)=(u_{0,\varepsilon}-\Tilde{u}_{0,\varepsilon})(x),~~~x\in \mathbb{R}^{d},
        \end{array}
    \end{equation}
    and $Y_{\varepsilon}$ solves
    \begin{equation}\label{Equation Y in uniqueness}
        \bigg\{
        \begin{array}{l}
        i\partial_{t}Y_{\varepsilon}(t,x;\tau) - (-\Delta)^{s}Y_{\varepsilon}(t,x;\tau) - V_{\varepsilon}(x)Y_{\varepsilon}(t,x;\tau) = 0,~~~(t,x)\in\left[\tau,T\right]\times \mathbb{R}^{d},\\
        Y_{\varepsilon}(\tau,x;\tau)= f_{\varepsilon}(\tau,x),~~~x\in\times \mathbb{R}^{d}.
        \end{array}
    \end{equation}
    We take the $L^2$-norm in \eqref{Duhamel represent. U in uniqueness}, and we take into consideration that $t\in [0,T]$, we get
    \begin{equation}\label{Duhamel solution estimate}
        \Vert U_{\varepsilon}(t,\cdot)\Vert_{L^2} \leq \Vert W_{\varepsilon}(t,\cdot)\Vert_{L^2} + \int_{0}^{t}\Vert Y_{\varepsilon}(t,\cdot;\tau)\Vert_{L^2}\d \tau,
    \end{equation}
    where we used Minkowski's integral inequality. We have
    \begin{equation*}
        \Vert W_{\varepsilon}(t,\cdot)\Vert_{L^2} = \Vert u_{0,\varepsilon} - \tilde{u}_{0,\varepsilon}\Vert_{L^2},
    \end{equation*}
    and
    \begin{equation*}
        \Vert Y_{\varepsilon}(t,\cdot;\tau)\Vert_{L^2} = \Vert f_{\varepsilon}(\tau,\cdot)\Vert_{L^2},
    \end{equation*}
    following from the conservation law \eqref{Estimate0_lemma1}. In order to estimate $\Vert f_{\varepsilon}(\tau,\cdot)\Vert_{L^2}$, we proceed as follows:
    \begin{align*}
        \Vert f_{\varepsilon}(\tau,\cdot)\Vert_{L^2} \leq & \underbrace{\big\Vert \big( V_{\varepsilon}(\cdot)-\tilde{V}_{\varepsilon}(\cdot)\big)\tilde{u}_{\varepsilon}(\tau,\cdot) \big\Vert_{L^2}}_{\rm{I}} + \underbrace{\big\Vert \big( g_{\varepsilon}(\cdot)-\tilde{g}_{\varepsilon}(\cdot) \big)|\tilde{u}_{\varepsilon}(\tau,\cdot)|^2 \tilde{u}_{\varepsilon}(\tau,\cdot) \big\Vert_{L^2}}_{\rm{II}}\\ & + \underbrace{\big\Vert g_{\varepsilon}(\cdot)\big( |u_{\varepsilon}(\tau,\cdot)|^2 u_{\varepsilon}(\tau,\cdot) - |\tilde{u}_{\varepsilon}(\tau,\cdot)|^2 \tilde{u}_{\varepsilon}(\tau,\cdot) \big) \big\Vert_{L^2}}_{\rm{III}}.
    \end{align*}
    \textbf{Estimate for $\rm{I}$:} we have
    \begin{equation}
        \big\Vert \big( V_{\varepsilon}(\cdot)-\tilde{V}_{\varepsilon}(\cdot)\big)\tilde{u}_{\varepsilon}(\tau,\cdot) \big\Vert_{L^2} \leq \Vert V_{\varepsilon} - \tilde{V}_{\varepsilon} \Vert_{L^{\infty}} \Vert \tilde{u}_{\varepsilon}(\tau,\cdot) \Vert_{L^2}.
    \end{equation}
    \textbf{Estimate for $\rm{II}$:} we have
    \begin{equation}
        \big\Vert \big( g_{\varepsilon}(\cdot)-\tilde{g}_{\varepsilon}(\cdot)\big)|\tilde{u}_{\varepsilon}(\tau,\cdot)|^2 \tilde{u}_{\varepsilon}(\tau,\cdot) \big\Vert_{L^2} \leq \Vert g_{\varepsilon} - \tilde{g}_{\varepsilon} \Vert_{L^{\infty}} \big\Vert |\tilde{u}_{\varepsilon}(\tau,\cdot)|^2 \tilde{u}_{\varepsilon}(\tau,\cdot) \big\Vert_{L^2}.
    \end{equation}
    To estimate the term $\big\Vert |\tilde{u}_{\varepsilon}(\tau,\cdot)|^2 \tilde{u}_{\varepsilon}(\tau,\cdot) \big\Vert_{L^2}$, first, we remark that $$\big\Vert |\tilde{u}_{\varepsilon}(\tau,\cdot)|^2 \tilde{u}_{\varepsilon}(\tau,\cdot) \big\Vert_{L^2} = \Vert \tilde{u}_{\varepsilon}(\tau,\cdot) \Vert_{L^6}^3.$$ Now, for $d<3s$, Gagliardo-Nirenberg-Sobolev inequality (Proposition \ref{Proposition GNS inequality}), with $r=0, q=6, s_1 =0, p_1 =2, s_2 =s, p_2 =2$, comes into play. We may estimate
    \begin{equation}
        \Vert \tilde{u}_{\varepsilon}(\tau,\cdot) \Vert_{L^6} \lesssim \big\Vert \tilde{u}_{\varepsilon}(\tau,\cdot) \big\Vert_{H^s}^{\frac{d}{3s}} \big\Vert \tilde{u}_{\varepsilon}(\tau,\cdot) \big\Vert_{L^2}^{1-\frac{d}{3s}}.
    \end{equation}
    \textbf{Estimate for $\rm{III}$:} we have
    \begin{align}\label{Estimate III}
        \big\Vert g_{\varepsilon}(\cdot)\big( |u_{\varepsilon}(\tau,\cdot)|^2 u_{\varepsilon}(\tau,\cdot) &- |\tilde{u}_{\varepsilon}(\tau,\cdot)|^2 \tilde{u}_{\varepsilon}(\tau,\cdot) \big) \big\Vert_{L^2}\nonumber\\ &\lesssim \Vert g_{\varepsilon} \Vert_{L^\infty} \big\Vert |u_{\varepsilon}(\tau,\cdot)|^2 u_{\varepsilon}(\tau,\cdot) - |\tilde{u}_{\varepsilon}(s,\cdot)|^2 \tilde{u}_{\varepsilon}(\tau,\cdot) \big\Vert_{L^2}.
    \end{align}
    The $L^2$-norm in the right hand side in \eqref{Estimate III} can be estimated as follows:
    \begin{align*}
        \big\Vert |u_{\varepsilon}(\tau,\cdot)|^2 u_{\varepsilon}(\tau,\cdot) &- |\tilde{u}_{\varepsilon}(\tau,\cdot)|^2 \tilde{u}_{\varepsilon}(\tau,\cdot) \big\Vert_{L^2} = \big\Vert \overline{u}_{\varepsilon}(\tau,\cdot)u_{\varepsilon}^2(\tau,\cdot) - \overline{\tilde{u}}_{\varepsilon}(\tau,\cdot) \tilde{u}_{\varepsilon}^2(\tau,\cdot) \big\Vert_{L^2}\nonumber\\
        & \leq \underbrace{\big\Vert \overline{u}_{\varepsilon}(\tau,\cdot)\big(u_{\varepsilon}(\tau,\cdot) + \tilde{u}_{\varepsilon}(\tau,\cdot)\big)\big(u_{\varepsilon}(\tau,\cdot) - \tilde{u}_{\varepsilon}(\tau,\cdot)\big) \big\Vert_{L^2}}_{\rm{A}} \\
        & \qquad + \underbrace{\big\Vert \tilde{u}_{\varepsilon}^2(\tau,\cdot)\big( \overline{u_{\varepsilon}(\tau,\cdot) - \tilde{u}_{\varepsilon}(\tau,\cdot)} \big) \big\Vert_{L^2}}_{\rm{B}}.
    \end{align*}
    The term $``\rm{A}"$, is estimated as follows
    \begin{equation*}
        A \leq \big\Vert u_{\varepsilon}(\tau,\cdot) \big\Vert_{L^\infty} \big\Vert u_{\varepsilon}(\tau,\cdot) + \tilde{u}_{\varepsilon}(\tau,\cdot) \big\Vert_{L^\infty} \big\Vert u_{\varepsilon}(\tau,\cdot) - \tilde{u}_{\varepsilon}(\tau,\cdot) \big\Vert_{L^2},
    \end{equation*}
    and we use the fact that in the case when $d<2s$, the embedding $H^s (\mathbb{R}^d) \subset L^{\infty}(\mathbb{R}^d)$ holds, to get
    \begin{equation*}
        A \lesssim \big\Vert u_{\varepsilon}(\tau,\cdot) \big\Vert_{H^s} \big\Vert u_{\varepsilon}(\tau,\cdot) + \tilde{u}_{\varepsilon}(\tau,\cdot) \big\Vert_{H^s} \big\Vert u_{\varepsilon}(\tau,\cdot) - \tilde{u}_{\varepsilon}(\tau,\cdot) \big\Vert_{L^2}.
    \end{equation*}
    For $``\rm{B}"$, we argue similarly. We obtain
    \begin{align*}
        \big\Vert \tilde{u}_{\varepsilon}^2(\tau,\cdot)\big( \overline{u_{\varepsilon}(\tau,\cdot) - \tilde{u}_{\varepsilon}(\tau,\cdot)} \big) \big\Vert_{L^2} & \lesssim  \big\Vert \tilde{u}_{\varepsilon}(\tau,\cdot) \big\Vert_{L^\infty}^2 \big\Vert u_{\varepsilon}(\tau,\cdot) - \tilde{u}_{\varepsilon}(\tau,\cdot) \big\Vert_{L^2}\\
        & \lesssim \big\Vert \tilde{u}_{\varepsilon}(\tau,\cdot) \big\Vert_{H^s}^2 \big\Vert u_{\varepsilon}(\tau,\cdot) - \tilde{u}_{\varepsilon}(\tau,\cdot) \big\Vert_{L^2}.
    \end{align*}
    We collect the obtained estimates, to get
    \begin{align*}
        \Vert U_{\varepsilon}(t,&\cdot)\Vert_{L^2} \lesssim \Vert u_{0,\varepsilon} - \tilde{u}_{0,\varepsilon}\Vert_{L^2} + \int_{0}^{t} \Bigg\{ \Vert V_{\varepsilon} - \tilde{V}_{\varepsilon} \Vert_{L^{\infty}} \Vert \tilde{u}_{\varepsilon}(\tau,\cdot) \Vert_{L^2} \\
        & + \Vert g_{\varepsilon} - \tilde{g}_{\varepsilon} \Vert_{L^{\infty}}\big\Vert \tilde{u}_{\varepsilon}(\tau,\cdot) \big\Vert_{H^s}^{\frac{d}{s}} \big\Vert \tilde{u}_{\varepsilon}(\tau,\cdot) \big\Vert_{L^2}^{3-\frac{d}{s}} \\
        & + \Vert g_{\varepsilon} \Vert_{L^\infty}\bigg[\big\Vert u_{\varepsilon}(\tau,\cdot) \big\Vert_{H^s} \big\Vert u_{\varepsilon}(\tau,\cdot) + \tilde{u}_{\varepsilon}(\tau,\cdot) \big\Vert_{H^s} \big\Vert u_{\varepsilon}(\tau,\cdot) - \tilde{u}_{\varepsilon}(\tau,\cdot) \big\Vert_{L^2}\\
        & + \big\Vert \tilde{u}_{\varepsilon}(\tau,\cdot) \big\Vert_{H^s}^2 \big\Vert u_{\varepsilon}(\tau,\cdot) - \tilde{u}_{\varepsilon}(\tau,\cdot) \big\Vert_{L^2}\bigg] \Bigg\}\d \tau.
    \end{align*}
    Now, by using the triangle inequality for the term $\big\Vert u_{\varepsilon}(\tau,\cdot) + \tilde{u}_{\varepsilon}(\tau,\cdot) \big\Vert_{H^s}$ and the fact that $(V_{\varepsilon}-\Tilde{V}_{\varepsilon})_{\varepsilon}$, $(g_{\varepsilon}-\Tilde{g}_{\varepsilon})_{\varepsilon}$, and $(u_{0,\varepsilon}-\Tilde{u}_{0,\varepsilon})_{\varepsilon}$ are $L^{\infty}(\mathbb{R}^d)$-negligible,$L^{\infty}(\mathbb{R}^d)$-negligible, and $L^{2}(\mathbb{R}^d)$-negligible, respectively, and the $L^\infty$-moderateness assumption on $(g_{\varepsilon})_{\varepsilon}$, and the $H^s$-moderateness of $(u_{\varepsilon})_{\varepsilon}$ and $(\tilde{u}_{\varepsilon})_{\varepsilon}$, being very weak solutions to \eqref{Equation in VW well-posedness}, we arrive at
    \begin{equation}\label{Estimate U uniqueness}
        \Vert U_{\varepsilon}(t,\cdot)\Vert_{L^2} \lesssim \varepsilon^{N} +  \int_{0}^{t} \omega(\varepsilon)^{-N_0}\big\Vert U_{\varepsilon}(\tau,\cdot) \big\Vert_{L^2} \d \tau,
    \end{equation}
    for every $t\in [0,T]$ and some $N_0 \in \mathbb{N}_0$, and any $N>0$. Gronwall's inequality (see Lemma \ref{Lemma: Gronwall}) applies. We get
    \begin{equation*}
        \Vert U_{\varepsilon}(t,\cdot)\Vert_{L^2} \lesssim \varepsilon^N \exp \big( \int_0^t \omega(\varepsilon)^{-N_0}\d \tau \big).
    \end{equation*}
    for all $t\in [0,T]$. If we choose $\omega(\varepsilon) = \big( \log(\varepsilon^{-1}) \big)^{\frac{-1}{N_0}}$, we obtain    
    \begin{equation*}
        \sup_{t\in [0,T]}\Vert U_{\varepsilon}(t,\cdot)\Vert_{L^2} \lesssim \varepsilon^{N-T},
    \end{equation*}
    for any $N>0$, proving the uniqueness of the very weak solution.
\end{proof}

\section{Compatibility with classical theory}
In this section, we prove that the notion of very weak solution is compatible with the classical theory, in the sense that, if a classical solution to 
\eqref{Equation in introduction} exists in the frame of Lemma \ref{Lemma1}, then the very weak solution recaptures it.

\begin{thm}\label{Compatibility theorem}
    Let $T>0$. Fix $d\geq 1$ and set $s>\frac{d}{2}$. Let $g,V\in L^{\infty}(\mathbb{R}^d)$ be non-negative and suppose that $u_0 \in H^{s}(\mathbb{R}^d)\cap L^4(\mathbb{R}^d)$. Then, for any regularising nets $(V_{\varepsilon})_{\varepsilon}=(V\ast\psi_{\varepsilon})_{\varepsilon}$ and $(g_{\varepsilon})_{\varepsilon}=(g\ast\psi_{\varepsilon})_{\varepsilon}$ for the equation coefficients, satisfying
    \begin{equation}\label{approx.condition1}
        \lVert V_{\varepsilon} - V\rVert_{L^{\infty}} \rightarrow 0,\quad\text{and}\quad \lVert g_{\varepsilon} - g\rVert_{L^{\infty}} \rightarrow 0,
    \end{equation} 
    as $\varepsilon \rightarrow 0$, and any regularising net $(u_{0,\varepsilon})_{\varepsilon}=(u_{0}\ast\psi_{\varepsilon})_{\varepsilon}$ for the initial function $u_0$, the net $(u_{\varepsilon})_{\varepsilon}$ converges to the classical solution (given by Lemma \ref{Lemma1}) of the Cauchy problem (\ref{Equation in introduction}) in $L^{2}$ as $\varepsilon \rightarrow 0$.
\end{thm}

\begin{proof}
    Let $(u_{\varepsilon})_{\varepsilon}$ denote the very weak solution from Theorem \ref{Existence theorem} and $u$ the classical one from Lemma \ref{Lemma1}. The very weak solution satisfies
    \begin{equation*}
        \bigg\{
        \begin{array}{l}
        i\partial_{t}u_{\varepsilon}(t,x) = (-\Delta)^{s}u_{\varepsilon}(t,x) + V_{\varepsilon}(x)u_{\varepsilon}(t,x) + g_{\varepsilon}(x)\vert u_{\varepsilon}(t,x)\vert^2 u_{\varepsilon}(t,x),\\
        u_{\varepsilon}(0,x)=u_{0,\varepsilon}(x),
        \end{array}
    \end{equation*}
    with $(t,x)\in [0,T]\times\mathbb{R}^d$, and the classical one solves
    \begin{equation*}
        \left\lbrace
        \begin{array}{l}
        iu_{t}(t,x) = (-\Delta)^{s} u(t,x) + V(x)u(t,x) + g(x)\vert u(t,x)\vert^2 u(t,x) ,~~~(t,x)\in\left[0,T\right]\times \mathbb{R}^{d},\\
        u(0,x)=u_{0}(x).
        \end{array}
        \right.
    \end{equation*}
    Let $U_{\varepsilon}(t,x):= u_{\varepsilon}(t,x) - u(t,x)$. Then, $U_{\varepsilon}$ is a solution to
    \begin{equation}\label{Equation U in compatibility}
        \bigg\{
        \begin{array}{l}
        i\partial_{t}U_{\varepsilon}(t,x) - (-\Delta)^{s}U_{\varepsilon}(t,x) - V_{\varepsilon}(x)U_{\varepsilon}(t,x) = f_\varepsilon (t,x),~~~(t,x)\in\left[0,T\right]\times \mathbb{R}^{d},\\
        U_{\varepsilon}(0,x)=(u_{0,\varepsilon}-u_0)(x),~~~x\in \mathbb{R}^{d},
        \end{array}
    \end{equation}
    where
    \begin{align*}
        f_\varepsilon (t,x):= & \big( V_{\varepsilon}(x)-V(x)\big)u(t,x) + \big( g_{\varepsilon}(x)-g(x) \big)|u(t,x)|^2 u(t,x)\\
        & + g_{\varepsilon}(x)\Big( |u_{\varepsilon}(t,x)|^2 u_{\varepsilon}(t,x) - |u(t,x)|^2 u(t,x) \Big).
    \end{align*}
    Following step by step the arguments in the proof of Theorem \ref{Uniqueness theorem}, we arrive at
    \begin{align*}
        \Vert U_{\varepsilon}(t,&\cdot)\Vert_{L^2} \lesssim \Vert u_{0,\varepsilon} - u_0\Vert_{L^2} + \int_{0}^{t} \Bigg\{ \Vert V_{\varepsilon} - V \Vert_{L^{\infty}} \Vert u(\tau,\cdot) \Vert_{L^2} \\
        & + \Vert g_{\varepsilon} - g \Vert_{L^{\infty}}\big\Vert u(s,\cdot) \big\Vert_{H^s}^{\frac{d}{s}} \big\Vert u(\tau,\cdot) \big\Vert_{L^2}^{3-\frac{d}{s}} \\
        & + \Vert g_{\varepsilon} \Vert_{L^\infty}\bigg[\big\Vert u_{\varepsilon}(\tau,\cdot) \big\Vert_{H^s}  \big\Vert u_{\varepsilon}(\tau,\cdot) + u(\tau,\cdot) \big\Vert_{H^s} \big\Vert u_{\varepsilon}(\tau,\cdot) - u(\tau,\cdot) \big\Vert_{L^2} \\
        & + \big\Vert u(\tau,\cdot) \big\Vert_{H^s}^2 \big\Vert u_{\varepsilon}(\tau,\cdot) - u(\tau,\cdot) \big\Vert_{L^2}\bigg] \Bigg\}\d \tau,
    \end{align*}
    for any $t\in [0,T]$. Since $g\in L^\infty$, one can easily see that the term $\Vert g_\varepsilon \Vert_{L^\infty}$ is bounded uniformly in $\varepsilon$. The net $(u_{\varepsilon})_{\varepsilon}$ is a classical solution to the family of regularised problems \eqref{Regularised problems}, it satisfies the estimate \eqref{Estimate1_lemma1} from Lemma \ref{Lemma1}. Thus, the terms $\Vert u_\varepsilon(\tau,\cdot) \Vert_{H^s}$ and $\Vert u_\varepsilon(\tau,\cdot) \Vert_{L^2}$ in the above estimate are bounded as well, uniformly in $t\in [0,T]$. Also, $\Vert u(\tau,\cdot) \Vert_{H^s}$ and $\Vert u(\tau,\cdot) \Vert_{L^2}$ are bounded, since $u$ is the classical solution to \eqref{Equation in introduction}. Now, by using the triangle inequality, one obtains
    \begin{equation}\label{Estimate U compatibility}
        \Vert U_{\varepsilon}(t,\cdot)\Vert_{L^2} \lesssim \Vert u_{0,\varepsilon} - u_0\Vert_{L^2} + \Vert V_{\varepsilon} - V \Vert_{L^{\infty}} + \Vert g_{\varepsilon} - g \Vert_{L^{\infty}} + \int_{0}^{t} \big\Vert U_{\varepsilon}(\tau,\cdot) \big\Vert_{L^2},
    \end{equation}
    for any $t\in[0,T]$. Once again, Gronwall's inequality comes into play. We obtain
    \begin{equation*}
        \Vert U_{\varepsilon}(t,\cdot)\Vert_{L^2} \lesssim \Big[ \Vert u_{0,\varepsilon} - u_0\Vert_{L^2} + \Vert V_{\varepsilon} - V \Vert_{L^{\infty}} + \Vert g_{\varepsilon} - g \Vert_{L^{\infty}} \Big] \exp{t},
    \end{equation*}
    for all $t\in [0,T]$. Now, by using the assumption that
    \begin{equation*}
        \lVert V_{\varepsilon} - V\rVert_{L^{\infty}} \rightarrow 0,\quad\text{and}\quad \lVert g_{\varepsilon} - g\rVert_{L^{\infty}} \rightarrow 0, ~~~\text{as}\,\, \varepsilon\to 0,
    \end{equation*}
    and that $\Vert u_{0,\varepsilon} - u_0\Vert_{L^2} \rightarrow 0$, as $\varepsilon \rightarrow 0$, this latter being a consequence of the fact that $C_c^{\infty}(\mathbb{R}^d)$ is dense in $L^2(\mathbb{R}^d)$, we arrive at
    \begin{equation*}
        \sup_{t\in [0,T]}\Vert U_{\varepsilon}(t,\cdot)\Vert_{L^2} \rightarrow 0, 
    \end{equation*}
    as $\varepsilon$ goes to $0$, ending the proof.
\end{proof}

\section{Numerical simulations}
In this section, we carry out numerical experiments of the cubic nonlinear 
Schr\"odi-nger equation in the one-dimensional case. 
In this work, we are interested in the singular cases of the coefficients $V(x), g(x)$. Even, we can allow them to be distributional, in particular, to have $\delta$-like or $\delta^2$-like singularities. As it was theoretically outlined in \cite{RT17a} and \cite{RT17b}, we start to analyse our problem by regularising distributions $V, g$  by a parameter $\varepsilon$, that is, we set
\begin{equation}
V_{\varepsilon}(x):=(V*\varphi_{\varepsilon})(x), \quad g_{\varepsilon}(x):=(g*\varphi_{\varepsilon})(x), 
\end{equation}
as the convolution with the mollifier
\begin{equation}
\varphi_{\varepsilon}(x)=\frac{1}{\varepsilon}\varphi(x/\varepsilon),
\end{equation}
where
$\varphi(x)=
c \exp{\left(\frac{1}{x^{2}-1}\right)}$ for $|x| < 1,$ and $\varphi(x)=0$ otherwise. Here  $c \simeq 2.2523$ to get
$\int\limits_{-\infty}^{\infty}  \varphi(x)dx=1.$

Then, instead of \eqref{Equation in introduction}, we consider the regularised problem for $s=1$

\begin{align}
   i\partial_{t}u_{\varepsilon}(t,x) 
   = (-\Delta)^{s}u_{\varepsilon}(t,x) + V_{\varepsilon}(x)u_{\varepsilon}(t,x) + g_{\varepsilon}(x)\vert u_{\varepsilon}(t,x)\vert^2 u_{\varepsilon}(t,x)=0,
\label{Equation_num2}
\end{align}
for $(t,x)\in (0,T)\times \left(0,10\right)$, with the Cauchy data
$$
u_{\varepsilon}(0,x)=u_{0}(x),
$$
for all $x\in\left[0,10\right].$
Here, we put
\[
u_0(x)=
\left\lbrace
\begin{array}{l}
\exp(\frac{1}{(x-5)^2+0.25}), \,\quad\quad|x - 5| < 0.5 \\
0,  \,\,\,\,\quad \quad \quad \quad \quad\quad\quad\quad  | x - 5| \geq 0.5,
 \end{array}
\right.
\]
For the coefficient pair $(V,g)$ we consider the following representative configurations, where
$\delta$ denotes the Dirac delta distribution and the singularity is supported at $x_0=4.5$:
\begin{itemize}
\item \textbf{Case 1 (regular):} $V(x)=1$, \quad $g(x)=1$.
\item \textbf{Case 2 (singular potential):} $V(x)=1+\delta(x-x_0)$, \quad $g(x)=1$.
\item \textbf{Case 3 (singular nonlinear coefficient):} $V(x)=1$, \quad $g(x)=1+\delta(x-x_0)$.
\item \textbf{Case 4 (both singular):} $V(x)=1+\delta(x-x_0)$, \quad $g(x)=1+\delta(x-x_0)$.
\end{itemize}

\paragraph{Case 1 (regular reference).}
As a baseline for comparison, we first illustrate the regular case $V(x)\equiv 1$ and $g(x)\equiv 1$.
Figure~\ref{fig1} shows the solution of $u(t,x)$ at $t=10$.

\begin{figure}[h]
\centering
\includegraphics[width=0.55\textwidth]{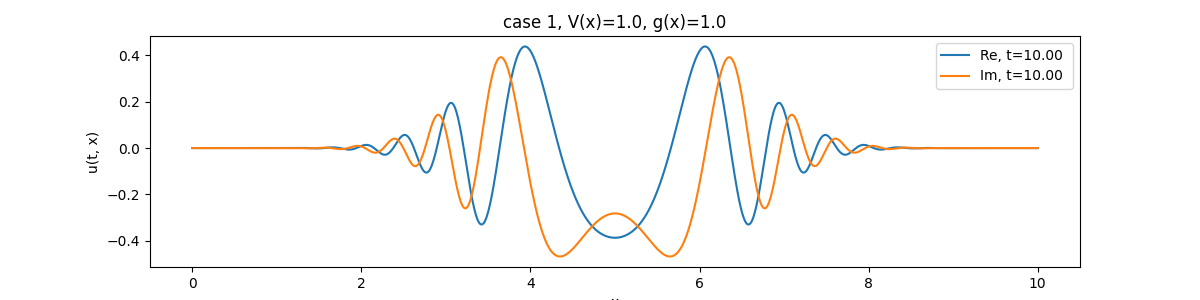}
\caption{Case 1: real and imaginary parts of $u(t,x)$ at $t=10$.}
\label{fig1}
\end{figure}

\paragraph{Case 2 (singular potential).}
Figure~\ref{fig2} depicts the evolution of the wave function $u(t,x)$ (solution of \eqref{Equation_num2})
for $\varepsilon\in\{1.0,0.7,0.3,0.01\}$ in Case~2, where the coefficient $V(x)$ contains a Dirac-type
singularity at $x=x_0$.
We plot both the real and imaginary parts. For these values of $\varepsilon$, the solution exhibits
a smooth propagation away from the singular point, with only a localised perturbation near $x=x_0$.

\begin{figure}[h]
\centering
\includegraphics[width=0.60\textwidth]{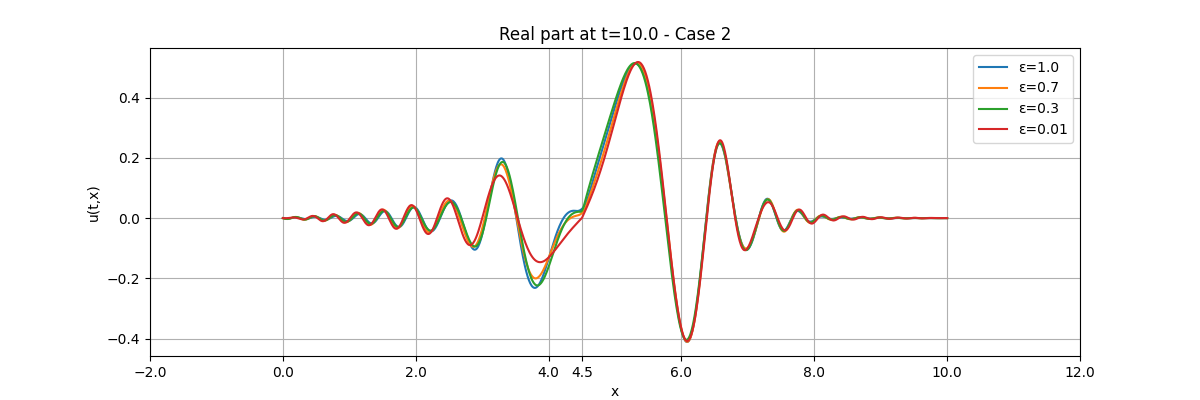}\\[2mm]
\includegraphics[width=0.60\textwidth]{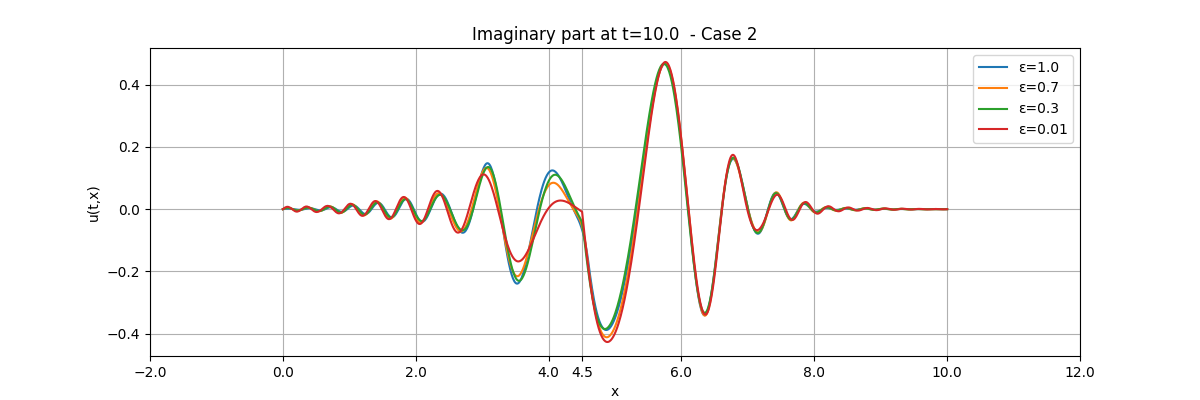}
\caption{Case~2: real and imaginary parts of $u(t,x)$ for $\varepsilon\in\{1.0,0.7,0.3,0.01\}$.}
\label{fig2}
\end{figure}

\paragraph{Case 3 (singular source).}
Figure~\ref{fig3} shows the corresponding results for Case~3, where the Dirac-type singularity is placed
in the coefficient $g(x)$ at $x=x_0$. Here we use $\varepsilon\in\{0.015,0.01,0.009,0.005\}$.
For larger $\varepsilon$, the regularisation smoothens the source contribution so strongly that its influence
on the solution is barely visible. Reducing $\varepsilon$ sharpens the localised forcing and makes the
impact of the singular source more pronounced.

\begin{figure}[h]
\centering
\includegraphics[width=0.60\textwidth]{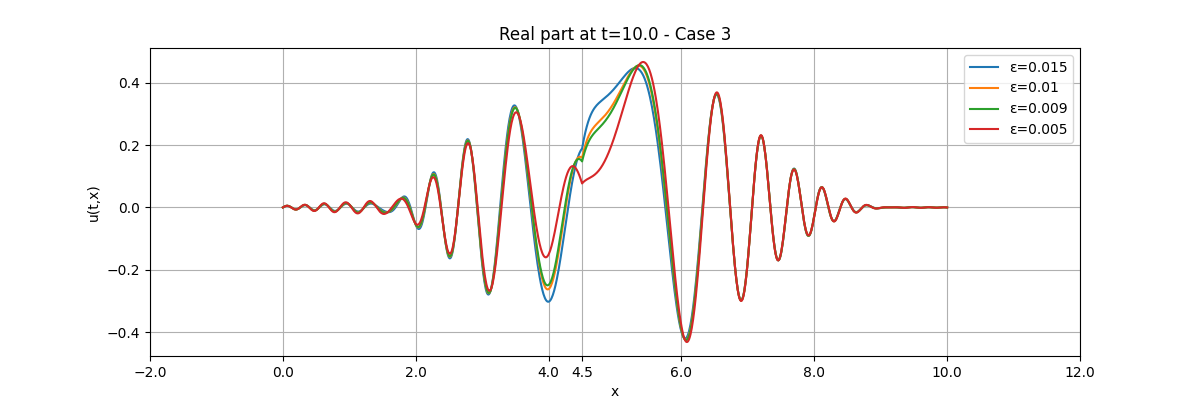}\\[2mm]
\includegraphics[width=0.60\textwidth]{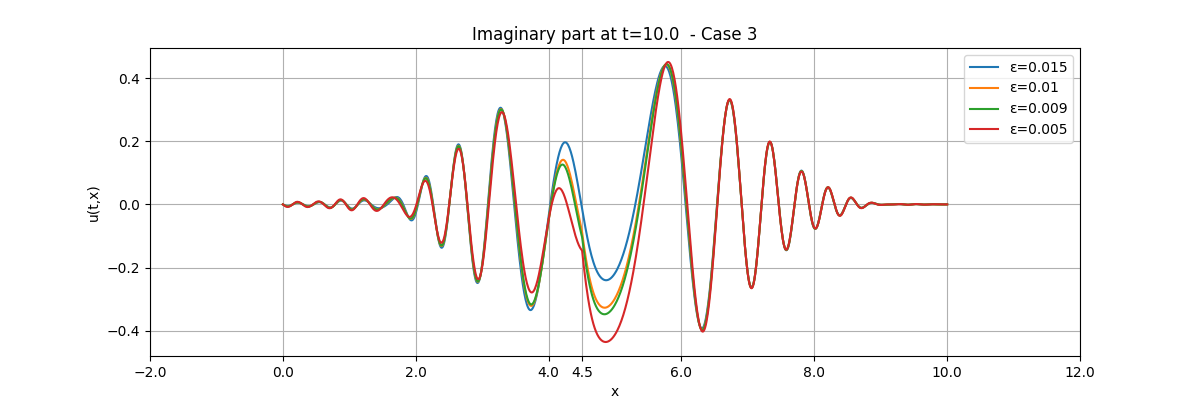}
\caption{Case~3: real and imaginary parts of $u(t,x)$ for $\varepsilon\in\{0.015,0.01,0.009,0.005\}$.}
\label{fig3}
\end{figure}

\paragraph{Case 4 (both coefficients singular).}
Finally, Figure~\ref{fig4} presents Case~4, in which both $V(x)$ and $g(x)$ contain Dirac-type
singularities supported at the same point $x=x_0$, for $\varepsilon\in\{1.0,0.7,0.3,0.01\}$.
As $\varepsilon$ decreases, the solution exhibits a stronger localised suppression near $x=x_0$; in particular,
the wave amplitude tends to vanish at the singular point, indicating a trapping (or blocking) effect.

\begin{figure}[h]
\centering
\includegraphics[width=0.60\textwidth]{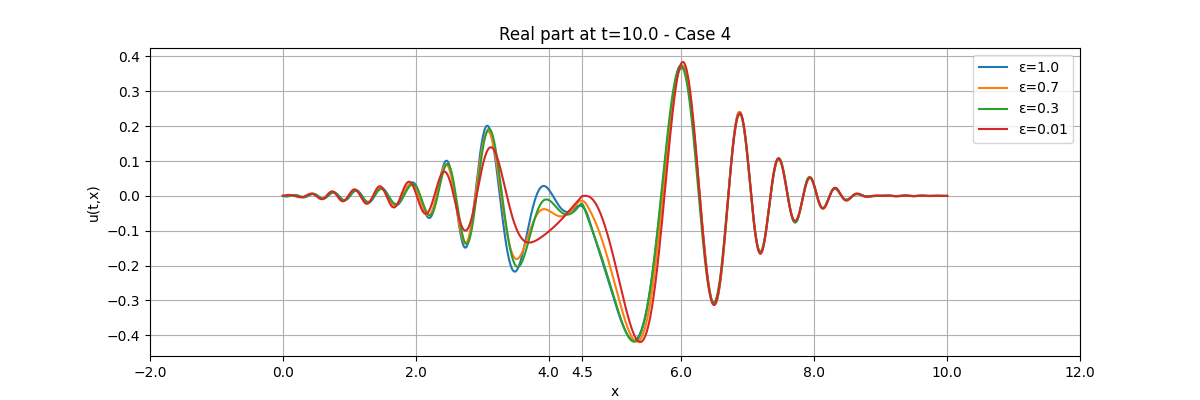}\\[2mm]
\includegraphics[width=0.60\textwidth]{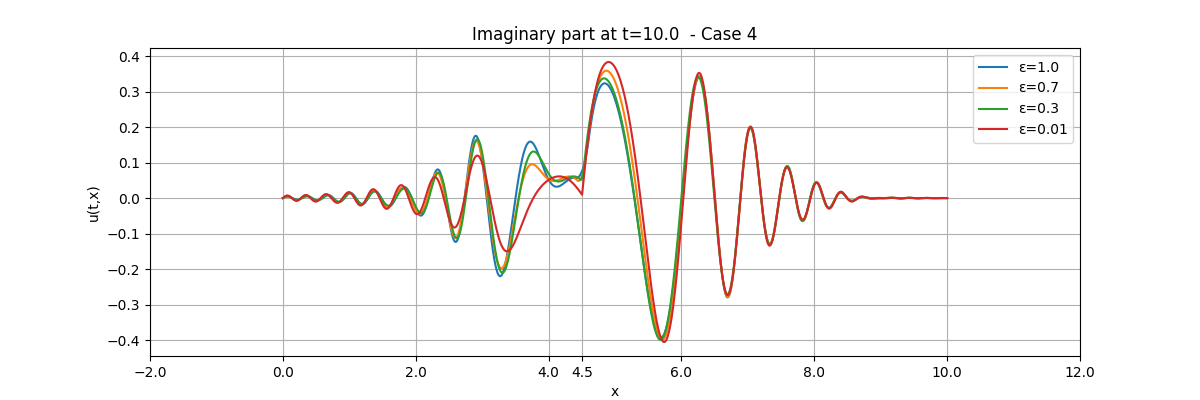}
\caption{Case~4: real and imaginary parts of $u(t,x)$ for $\varepsilon\in\{1.0,0.7,0.3,0.01\}$.}
\label{fig4}
\end{figure}

\paragraph{Discussion.}
These experiments demonstrate how Dirac-type terms in the coefficients can qualitatively alter wave propagation.
They also highlight the role of the regularisation parameter $\varepsilon$, which controls the sharpness of the
singular features and therefore the strength of the localised interaction. Overall, the numerical results
support the theoretical analysis and indicate that distributional coefficient models can be used to represent
highly localised events in wave transfer. This framework can be extended to more complex geometries, coupled
models, and related inverse problems.

\section{Concluding remarks}
We conclude this article with some remarks and open questions.
\begin{itemize}
    \item The novelty in this article lies in the fact that we are using the relatively recently introduced concept of very weak solutions, for nonlinear equations. Moreover, by incorporating singular objects, we are considering equations that cannot be posed in the framework of the classical notions of solutions (strong, weak or distributional). 
    \item Our uniqueness and compatibility results are proved in the case when $d<2s$, where the embedding $H^s(\mathbb{R}^d)\subset L^{\infty}(\mathbb{R}^d)$ holds, allowing to control $L^{\infty}$-norms by $H^s$-norms, nevertheless, many of our arguments remain valid in the general case. This gap will be addressed in a forthcoming paper.
    \item An easy extension of the results obtained in this paper is to consider the non-homogeneous case for the Cauchy problem \eqref{Equation in introduction}. The key ingredient is Duhamel's principle.
    \item For the sake of simplicity, this paper dealt with a Schr\"odinger-type equation with a cubic nonlinearity. The results obtained here can be easily extended for nonlinearities of the form $|u(t,x)|^{p-1}u(t,x)$ for an integer $p$, by reasoning exactly along the same lines. Furthermore, one can consider equations with combined power-type nonlinearities, that is, equations of the form
    \begin{equation*}
        iu_{t} = (-\Delta)^{s} u + Vu + g_1\vert u\vert^{p_1} u + g_2\vert u\vert^{p_2} u.
    \end{equation*}
    The question that naturally arises here, is to discuss the very weak well-posedness of such problems according to the values of parameters appearing in the equation. Moreover, general forms of nonlinearities $N(u)$ could also be considered. In order to prove energy estimates and uniqueness results, assumptions such as
    \begin{equation*}
        N(u) \leq C P(|u|),
    \end{equation*}
    and
    \begin{equation*}
        |N(u)-N(v)| \leq M(|u|,|v|) |u-v|,
    \end{equation*}
    where $P$ is a polynomial, and $M$ depends polynomially on $|u|$ and $|v|$, should be taken.
\end{itemize}







\end{document}